\documentclass[12pt]{article}
\usepackage{latexsym}
\usepackage{epsfig}
\usepackage{amsmath,amsthm,amssymb,enumerate}
\usepackage[active]{srcltx}
\usepackage[noend]{algorithmic}
\usepackage[hidelinks]{hyperref}
\usepackage{algorithm}
\usepackage{color}

\definecolor{plum}{rgb}{0.62, 0.0, 0.77}
\allowdisplaybreaks

\newcounter{rot}

\def\a{\alpha}   
\def\e{\varepsilon} \def\f{\phi}   
\def\G{\Gamma}  
     \def\l{\lambda}
 \def\m{\mu} \def\n{\nu} 
  \def\s{\sigma} 
\def\t{\tau}   \def\U{\Upsilon}

\newtheorem{theorem}{Theorem}
\newtheorem{lemma}[theorem]{Lemma}

\newtheorem{claim}{Claim}

\newcommand{\rdup}[1]{{\left\lceil #1 \right\rceil }}
\newcommand{\rdown}[1]{{\left\lfloor #1\right \rfloor}}

\newcommand{\brac}[1]{\left(#1\right)}

\newcommand{\bfrac}[2]{\left(\frac{#1}{#2}\right)}

\def\cE{{\cal E}}

\newcommand{\set}[1]{\left\{#1\right\}}

\def\E{\mbox{{\bf E}}}

\def\Pr{\mbox{{\bf Pr}}}

\newcommand{\ignore}[1]{}

\newcommand{\cA}{{\cal A}}

\newcommand{\card}[1]{\left|#1\right|}

\newcommand{\beq}[2]{\begin{equation}\label{#1}#2\end{equation}}

\def\wW{\widehat{W}}

\newcommand{\multstar}[1]{\begin{multline*}#1\end{multline*}}
\newcommand{\mult}[2]{\begin{multline}\label{#1}#2\end{multline}}

\author{Alan Frieze\thanks{Research supported in part by NSF grant DMS1661063
}\qquad Wesley Pegden\thanks{Research supported in part by NSF grant DMS1363136}\qquad Tomasz Tkocz\\Department of Mathematical Sciences\\Carnegie Mellon University\\Pittsburgh PA15217\\U.S.A.}

\begin{document}
\title{On random multi-dimensional assignment problems}
\maketitle
\begin{abstract}
We study random multidimensional assignment problems where the costs decompose into the sum of independent random variables. In particular, in three dimensions, we assume that the costs $W_{i,j,k}$ satisfy $W_{i,j,k}=a_{i,j}+b_{i,k}+c_{j,k}$ where the $a_{i,j},b_{i,k},c_{j,k}$ are independent uniform $[0,1]$  random variables. Our objective is to minimize the total cost and we show that w.h.p. a simple greedy algorithm is a $(3+o(1))$-approximation. This is in contrast to the case where the $W_{i,j,k}$ are independent  exponential rate 1 random variables. Here all that is known is an $n^{o(1)}$-approximation, due to Frieze and Sorkin.
\end{abstract}
\section{Introduction}
The (planar) three dimensional assignment problem is a natural generalisation of the classical assignment ptoblem. As an optimization problem it can be expressed as follows: we are given real values $W_{i,j,k}$ for $i,j,k\in [n]$ and we are asked to 
\[
Minimize \set{\sum_{i=1}^nW_{i,\s(i),\t(i)}:\;\s,\t\text{ are permutations of }[n]}.
\]
This is an NP-hard problem and occurs for example as a practical problem \cite{FY}. In this paper we study the following simple greedy heuristic:
\begin{algorithm}[H]
\caption{{\sc Greedy}($m$)}
\label{greedy2}
\begin{algorithmic}[1]
\STATE 
Let $B:=C:=[n]$, and $T:=\emptyset$.
\FOR{$i=1,\ldots,m$}
\STATE Let $W_{i,j,k} =\min\set{W_{i,j',k'} \colon j'\in B,k'\in C}$.
\STATE Add $(i,j,k)$ to $T$ and remove $j$ from $B$ and $k$ from $C$.
\ENDFOR
\STATE Return the set of triples in $T$ as a partial assignment.
\STATE Complete the assignment with one of the remaining $(n-m)!^2$ possibilities.
\end{algorithmic}
\end{algorithm}

Several authors have considered the average case where the $W_{i,j,k}$ are random variables. Kravtsov \cite{Krav1} considered the case where the $W_{i,j,k}$ are chosen randomly from $[1,M]$ where $M=n^\a$ for some $\a<1$. Here the minimum is at least $n$ and it is not difficult to show (see Section \ref{Easycase} that with the choice of $m=n-\log n$ that w.h.p. (i) {\sc greedy}($m$) runs in polynomial time and (ii) it outputs a solution of value $n+o(n)$. In this case Step 6 can be completed via the choice of an arbitrary completion.

It is more difficult to analyse the case where $L\gg n$ and the case where the $W_{i,j,k}$ are independent exponential rate 1 random variables is (essentially) a scaled version of such a case. This case was considered by Frieze and Sorkin \cite{FS} and they proved the following theorem.
\begin{theorem}[Frieze and Sorkin]
Suppose that the $W_{i,j,k}$ are independent EXP(1) random variables and that $Z_n$ denote the value of the optimum. Then (a) $\frac{1}{n}\leq \E(Z_n)=O\bfrac{\log n}{n}$ and (b) there is a polynomial time algorithm that w.h.p. finds a solution of value $\frac{1}{n^{1-o(1)}}$.
\end{theorem}
This is where the problem stands for such $W_{i,j,k}$ and here we consider the case where
\beq{costs}{
W_{i,j,k}=a_{i,j}+b_{i,k}+c_{j,k},\,1\leq i,j,k\leq n,
}
where the $a_{i,j},b_{i,k},c_{j,k}$ are independent uniform $[0,1]$ random variables.

We note that the problem considered in \cite{FY} was of the form given in \eqref{costs}. We will prove the following theorem.
\begin{theorem}\label{th1}
There exist constants $c_1,c_2$ such that (a) $\E(Z_n)\geq c_1n^{1/3}$ and (b) {\sc greedy}($n-n^{1/4}$) finds a solution of expected value at most $c_2n^{1/3}$. In this case Step 6 can be completed by choosing an arbitrary completion.
\end{theorem}
Before giving a proper proof, we give a heuristic argument for (a). Fix $i$ and consider $W_{i,j,k}$. For $W_{i,j,k}$ to be of order $n^{-\a}$ say we need each of 3 uniform $[0,1]$ varables to be of order $n^{-\a}$. This happens with probability $O(n^{-3\a})$ and there are $n^2$ choices and $3\a=2$ gives the largest value for $\a$. Summing over $i$ gives (a).

We discuss the rigorous proof of Theorem \ref{th1} in Section \ref{Proof} and in Section \ref{Higher} we consider the extension to higher dimensions.
\subsection{Preliminaries}
We sometimes refer to the Hoeffding bounds for the $S=S_1+S_2+\ldots+S_N$ where $S_1,S_2,\ldots,S_N\in [0,1]$ are independent and $\E(S_1)+\E(S_2)+\cdots+\E(S_N)=N\m$:
\beq{Chern}{
\Pr(|S-N\m|\geq \e N\m)\leq 2e^{-\e^2N\m/3}.
}
We say that a sequence of events $\cE_n$ occur {\em quite surely} if $\Pr(\neg\cE_n)=O(n^{-K})$ for any constant $K>0$.
\section{Proof of Theorem \ref{th1}}\label{Proof}
We begin by analysing the distribution of the smallest weight element in $\Pi_i=\set{i}\times [n]^2$.
\subsection{Weights in a fixed plane}\label{wts}
Let 
\[
W_n=\min\set{a_i+b_j+c_{i,j}:\text{$a_i,b_j,c_{i,j},\,i,j\in[n]$ are independent uniform $[0,1]$ random variables}}.
\]
\begin{lemma}\label{lem1}
$\E(W_n)\approx c_1n^{-2/3}$, where $c_1=6^{1/3}\G(3/4)$, where $\G$ denotes Euler's Gamma function.
\end{lemma}
\begin{proof}
Let
\beq{defL}{
L=\log n,\; I=\set{i:a_i\leq \frac{L}{n^{2/3}}},\; J=\set{j:b_j\leq \frac{L}{n^{2/3}}},\; X=\set{(i,j)\in I\times J:c_{i,j}\leq \frac{L}{n^{2/3}}}.
}
It follows from \eqref{Chern} that 
\beq{conc1}{
|I,|J|\in \left[\frac12Ln^{1/3},\frac32Ln^{1/3}\right]\quad q.s.
}
Conditional on the sizes of $I,J$ we have $|X|$ is distributed as $B(|I|\cdot|J|,L/n^{2/3})$. It follows from \eqref{Chern} that
\beq{sizeX}{
|X|\in \left[\frac{L^3}{8},10L^3\right]\quad q.s.
}
Thus let $\cE_L$ denote the even that $|X|\in \left[\frac{L^3}{8},10L^3\right]$.

Let $\cE_M$ denote the event that the edges in $X$ {\bf almost} form a matching. By this we mean that the graph induced by $X$ consists of a matching $M$ plus at most 4 extra edges $Y$. Then, 
\beq{ExZ}{
\E(W_n\mid\cE_M)\Pr(\cE_M)\leq \E(W_n)\leq \E(W_n\mid\cE_M)+3\Pr(\neg\cE_M).
}
We first deal with $\Pr(\neg\cE_M)$ by showing that.
\beq{notEM}{
\Pr(\neg\cE_M)=O\bfrac{L^{15}}{n}.
}
Let
\[
p=\frac{L}{n^{2/3}}.
\] 
Condition on $I,J$ satisfying \eqref{conc1}. Let $\G_X$ be the graph induced by $X$ and note that it is distributed as the binomial random graph $G_{|I|,|J|,p}$. 
\begin{claim}\label{cl1}
The following holds with probability $1-O(L^{15}/n)$: (i) $\G_X$ has no component with 4 or more edges and (ii) $\G_X$ has at most one component with 3 edges and (iii) $\G_X$ has at most 2 components with 2 edges.
\end{claim}
{\bf Proof of claim:} Let $K=|I|+|J|$.
\[
\Pr(\neg(i))=O\brac{\binom{K}{5}p^4}=O\bfrac{L^9n^{5/3}}{n^{8/3}}=O\bfrac{L^9}{n}.
\]
\[
\Pr(\neg(ii))=O\brac{\brac{\binom{K}{4}p^3}^2}=O\bfrac{L^{14}n^{8/3}}{n^{4}}= O\bfrac{L^{14}}{n^{4/3}}.
\]
\[
\Pr(\neg(iii))=O\brac{\brac{\binom{K}{3}p^2}^3}=O\bfrac{L^{15}n^{3}}{n^{4}}= O\bfrac{L^{15}}{n}.
\]
{\bf End of proof of claim.}

Now given $\cE_M$ we let $\wW_n$ denote the minimum weight in $M$ and we see that $\wW_n$ is the minimum of $|M|$ independent copies of $U=(U_1+U_2+U_3)p$ where $U_1,U_2,U_3$ are independent uniform $[0,1]$. 

Thus 
\[
\f(u)=\Pr(U\geq pu)=1-\frac16\sum_{k=0}^{\rdown{u}}(-1)^k\binom{3}{k}(u-k)^3.
\]
It follows that 
\begin{align}
\E(\wW_n\mid \cE_L,\cE_M,|M|)&=p\int_{u=0}^3\Pr(\wW_n\geq up\mid\cE_L,\cE_M,|M|)du\nonumber\\
&=p\int_{u=0}^3\f(u)^{|M|}du\nonumber\\
&=p(I_1+I_2+I_3),\label{I1I2I3}
\end{align}
where
\begin{align}
I_1&=\int_{u=0}^1\brac{1-\frac{u^3}{6}}^{|M|}du\label{f1}\\
&=\int_{u=0}^{1/L^{2/3}}\brac{1-\frac{u^3}{6}}^{|M|}du+ \int_{u=1/L^{2/3}}^1\brac{1-\frac{u^3}{6}}^{|M|}du \label{f2}\\
&=\int_{u=0}^{1/|L^{2/3}}\exp\set{-|M|u^3/6+O(|M|u^6)}du+O(e^{-\Omega(|M|/L^2)})\nonumber\\
&=\brac{1+O\bfrac{1}{L^2}}\int_{u=0}^{1/|L^{2/3}}e^{-|M|u^3/6}du+O(e^{-\Omega(|M|/L^2)})\nonumber\\
&=\brac{1+O\bfrac{1}{L^2}}\int_{u=0}^{\infty}e^{-|M|u^3/6}du\nonumber\\
&=\frac{\brac{1+O(L^{-2}})}{|M|^{1/3}}\int_{x=0}^\infty e^{-x^3/6}dx\nonumber\\
&=\frac{\brac{6^{1/3}\G(4/3)+O(L^{-2})}}{|M|^{1/3}}.\nonumber
\end{align}
Now because $\f(u)$ dereases monotonically with $u$ we have
\[
I_2=\int_{u=1}^2\f(u)^{|M|}du\leq \bfrac{5}{6}^{|M|}\text{ and }I_3=\int_{u=2}^3\f(u)^{|M|}du\leq \bfrac{5}{6}^{|M|}.
\]
Thus,
\beq{E1}{
\E(\wW_n\mid \cE_L,\cE_M,|M|)=\frac{\brac{6^{1/3}\G(4/3)+O(L^{-2})}}{|M|^{1/3}}p.
}
Integrating $|M|$ from \eqref{E1} we obtain
\beq{E2}{
\E(\wW_n\mid \cE_L,\cE_M)=\brac{6^{1/3}\G(4/3)+O(L^{-2})}\times \E((Bin(|I|\cdot|J|,p)-O(1))^{-1/3})\times p.
}
Given $\cE_L$ we see that the binomial is q.s. much greater than 4. Now, for $Nq$ large, we have, from \eqref{Chern}, that for $\e>0$,
\begin{align}
\E((Bin(N,q)-O(1))^{-1/3})&=\sum_{k=5}^N\binom{n}{k}q^k(1-q)^{N-k}(k-O(1))^{-1/3}\nonumber\\
&=\sum_{k=(1-\e)Nq}^{(1+\e)Nq}\binom{n}{k}q^k(1-q)^{N-k}(k-O(1))^{-1/3}+2e^{-\e^2Nq/3}\nonumber\\
&=\frac{1+O(\e)}{(Nq)^{1/3}}+O(e^{-\e^2Nq/3}),\label{E4}
\end{align}
provided $\e^2Nq\gg\log Nq$.

It then follows from \eqref{E2} that
\beq{E3}{
E(\wW_n\mid \cE_L,\cE_M)\approx \frac{6^{1/3}\G(4/3)p}{(|I|\cdot|J|\cdot p)^{1/3}}
}
Arguing as for \eqref{E4} and using the independence and concentration of $|I|,|J|$ around $Ln^{1/3}$, we see that
\beq{E5}{
E(\wW_n\mid \cE_M)\approx \frac{6^{1/3}\G(4/3)}{n^{2/3}}.
}
We now have to deal with the at most 4 edges in $Y$, since $W_n=\min\set{\wW_n,Z}$ where $Z$ is the minimum of at most 4 copies of $(U_1+U_2+U_3)p$, where $U_1,U_2,U_3$ are i.i.d. $U[0,1]$. Clearly $\E(W_n)\leq \E(\wW_n)$ and we need to argue that it is not much smaller. So, let $\cA=\set{\wW_n\leq pL^{-1/2}\leq Z}$. Now we have $\Pr(\cA)=1-O(L^{-1/2})$ and $\E(W_n)\geq \E(\wW_n\mid \cA)\Pr(\cA)$ and so we only have to verify now that $\E(\wW_n\mid\cA)$ is asymptotically equal to $\E(\wW_n)$. Now because $\wW_n$ and $Z$ are independent, we have, given $|M|$,
\begin{align}
\E(\wW_n\mid \cA)&=\E(\wW_n\mid \wW_n\leq pL^{-1/2}) =   \frac{1}{\Pr(\wW_n\leq pL^{-1/2})}\int_{u=0}^{pL^{-1/2}}\Pr(pL^{-1/2}\geq \wW_n\geq u) du\nonumber\\
&= \frac{1}{\Pr(\wW_n\leq pL^{-1/2})}\int_{u=0}^{pL^{-1/2}}\Pr(\wW_n\geq u) du - pL^{-1/2}\frac{\Pr(\wW_n > pL^{-1/2})}{\Pr(\wW_n\leq pL^{-1/2})}.\label{TT1}
\end{align}
Now  
\[
\Pr(\wW_n > pL^{-1/2}) = \brac{1-\frac{(L^{-1/2})^3}{6}}^{|M|} \leq e^{-|M|L^{-1/6}/6}.
\]
Furthermore,
\[
\Pr(\wW_n\geq u)\geq \brac{1-\frac{u^3}{6}}^{|M|}
\]
and so integral in the first term of \eqref{TT1} is at least
\[
p\int_{p=0}^{L^{-1/2}}\brac{1-\frac{u^3}{6}}^{|M|}du.
\] 
Thus
\[
\E(\wW_n\mid \cA)\geq (1-o(1))p\int_{p=0}^{L^{-1/2}}\brac{1-\frac{u^3}{6}}^{|M|}du- e^{-|M|L^{-1/6}/6}
\]
and we can proceed as for our estimation of $I_1$.

The lemma now follows after applying \eqref{ExZ} and \eqref{notEM}.
\end{proof}
This proves Part (a) of Theorem \ref{th1}, since clearly, $\E(Z_n)\geq n\E(W_n)$.
\subsection{Analysis of Greedy}
Let now $W_m$ denote the the weight of the triple $(i,j,k)$ added in the $m$th round of greedy.
\begin{lemma}\label{lem2}
If $m\leq n-n^{1/4}$ then 
\beq{EWell}{
\E(W_m) \lesssim c_1(n-m+1)^{-2/3}.
}
\end{lemma}
\begin{proof}
We let $I_m,J_m$ be as defined in \eqref{defL}, where we replace $n$ in the definition by $\n_m=n-m+1$. We keep $L$ as $\log n$ though and replace $p$ by $p_m=\frac{L}{\n_m}$. The values $a_{m,j},b_{m,k}$ are independent of the first $m-1$ rounds of {\sc greedy}. Now $|I|_m,|J_m|$ are distributed as $Bin(\n_m,L\n_m^{-2/3})$ and equation \eqref{Chern} implies that \eqref{conc1} holds q.s. with $n$ replaced by $\n_m$. Next define $X_m$ iteratively via $X_0=\emptyset$ and 
\[
X_m=\set{(i,j)\in (I_m\times J_m)\setminus \bigcup_{l<m}X_l,\,c_{i,j}\leq \frac{L}{\n_m^{2/3}}}.
\]
We will show below that 
\beq{Xell}{
\Pr\brac{\card{(I_m\times J_m)\cap \bigcup_{l<m}X_l}\geq 400L\n_m^{1/3} }=o(n^{-3}). 
}
Observe that $c_{i,j}$ for $(i,j)\in X_m$ is unconditioned by the history of {\sc greedy} to this point. Indeed, we will not have needed to expose its value in order to compute the sequence $W_1,W_2,\ldots,$ $W_{m-1}$. But if \eqref{Xell} holds then the analysis of Section \ref{wts} implies that
\[
\E(W_m)\approx c_1\n_m^{-2/3}.
\]
Indeed, going back to \eqref{E2} we 
\[
\text{replace $\E(Bin(|I|\cdot|J|,p)^{-1/3})$ by $\E((Bin(|I_m|\cdot|J_m|-400L\n_m^{1/3} ,p_m))^{-1/3})$}
\]
and continue as before.

It remains to verify \eqref{Xell}. Thus let $Y_m=(I_m\times J_m)\cap \bigcup_{l<m}X_l$ and $Z=|Y_m|$. Now the sequence of choices $I_\ell,J_\ell,\ell\leq m$ are independent and then for $(x,y)\in I_m\times J_m$ and $\ell<m$ we have
\mult{one}{
\Pr((x,y)\in I_\ell\times J_\ell\mid (\ell,x,y)\text{ not added to $T$ in Step 4})\leq\\ \frac{\Pr((x,y)\in I_\ell\times J_\ell)}{\Pr((\ell,x,y)\text{ not added to $T$ in Step 4})} \leq \frac{\n_\ell^{-4/3}}{1-o(\n_\ell^{-2})}.
}
It follows (using \eqref{conc1}) that 
\beq{EZ}{
\E(Z)\leq 4L^2\n_m^{2/3}\sum_{\ell=1}^{m-1}\frac{1}{\n_\ell^{4/3}}\leq 13L^2\n_m^{1/3}.
}
Unfortunately, this is not good enough to prove \eqref{Xell}. Instead, suppose that $S=\set{(x_i,y_i),i\in [s]}\subseteq I_m\times J_m$ where $s=O(1)$ and $S$ is a matching. Then, 
\multstar{
\Pr(S\subseteq Y_m)\leq \sum_{i_1\leq \cdots\leq i_{s}} \Pr\brac{\bigcap_{t=1}^s\set{(x_t,y_t)\in X_{i_t}}}=\\
 \sum_{i_1\leq \cdots\leq i_{s}}\prod_{t=1}^s\Pr\brac{{(x_t,y_t)\in X_{i_t}}\bigg| \bigcap_{\t=1}^{t-1}\set{(x_\t,y_\t)\in X_{i_\t}}}\leq \sum_{i_1\leq \cdots\leq i_{s}}\prod_{t=1}^s((1+o(1))\n_{i_t}^{-4/3})\\
\leq \prod_{t=1}^s\sum_{l=1}^{m}\frac{1+o(1)}{(n-l+1)^{4/3}}\leq \bfrac{4}{(n-m)^{1/3}}^s.
}
Thus,
\beq{compl}{
\Pr(\exists\text{ matching } S,|S|=s\mid \eqref{conc1})\leq \binom{10L^3}{s}\bfrac{4}{n^{1/12}}^s=o(n^{-3})
}
if $s=40$. Finally observe that if the maximum size of $S=s\leq 40$ and $|I_m|,|J_m|\leq 10L\n_m^{1/3}$ then $|Y_m|\leq s(|I_m|+|J_m|)\leq 10sL\n_m^{1/3}$ and the condition in \eqref{Xell} holds. 
\end{proof}
Given Lemma \ref{lem2} we see that the expected cost of the assignment produced by {\sc greedy} is at most
\beq{greedycost}{
(c_1+o(1))\sum_{m=1}^{n-n^{1/4}}\frac{1}{(n-m+1)^{2/3}}+n^{1/4}\approx 3c_1n^{1/3}.
}
The final $n-n^{1/4}$ steps cost at most $3$ per step and this completes the proof of Theorem \ref{th1}.
\section{Higher Dimensions}\label{Higher}
Consider for example 4 dimensions. Here we have two reasonable options.
\begin{enumerate}
\item $W_{i,j,k,l}=a_{i,j}+b_{i,k}+c_{i,l}+d_{j,k}+e_{j,l}+f_{k,l}$.
\item $W_{i,j,k,l}=a_{i,j,k}+b_{i,j,l}+c_{i,k,l}+d_{j,k,l}$.
\end{enumerate}
We have not considered the first option. The second option is a strightforward generalisation of what we have done so far. Here we will sketch a proof as a series of bullet points that the optimum and the greedy solution for the $d$-dimensional problem grow at rate $n^{1/d}$ in expectation. By the $d$-dimensional problem we mean
\[
Minimize \set{\sum_{i=1}^nW_{i,\s_1(i),\ldots,\s_{d-1}(i)}:\;\s_1,\ldots,,\s_{d-1}\text{ are permutations of }[n]}.
\]
where
\[
W_{i_1,\ldots,i_d}=\sum_{j=1}^dA^{(j)}_{i_1,\ldots,i_{j-1},i_{j+1},\ldots,i_d}\text{ is the sum of independent uniform }[0,1]\text{ random variables}.
\]
We claim that Theorem \ref{th1} can be generalised to 
\begin{theorem}\label{thd}
Suppose that $d\geq 3$. Then there exist constants $c_d,C_d$ such that (a) $\E(Z_n)\geq c_dn^{1/d}$ and (b) {\sc greedy}($n^{1/(d+1)}$) finds a solution of expected value at most $C_dn^{1/d}$. In this case Step 6 can be completed by choosing an arbitrary completion.
\end{theorem}
{\bf Proof Sketch:}\\
We can follow the argument in Lemma \ref{lem1} essentially replacing $n^{1/3}$ by $n^{1/d}$ and $n^{2/3}$ by $n^{(d-1)/d}$. In effect,  we make the following replacements:
\begin{enumerate}[{\bf (a):}]
\item $p$ becomes $L/n^{(d-1)/d}$.
\item $I,J$ will be replaced by $I_1,\ldots,I_{d-1}$ of expected size $np$.
\item In which case $X$ becomes $\set{{\bf j}\in (I_1\times\cdots\times I_{d-1}):W_{i,{\bf j}}\leq p}$.
\item \eqref{sizeX} becomes $|X|=\Theta(L^d)$.
\item A matching now means a matching in a $(d-1)$-uniform hypergraph $H$ induced by $I_1\times\cdots\times I_{d-1}$. In the proof of Claim \ref{cl1}, we now let $K=|I_1|+\cdots+|I_{d-2}|$. We now claim that with probability $\frac{1}{n^{1-o(1)}}$ there are at most $\frac{d}{\ell-1}$ components of $H$ with $\ell\leq d+1$ edges and no components with $d+2$ or more edges. Indeed, the probability that there are $a$ components of $H$ with $\ell$ edges can be bounded by
\[
\brac{\binom{K}{\ell+1+\ell(d-3)}p^{\ell}}^a=O\bfrac{L^{1+(d-1)\ell}}{n^{(\ell-1)/d}}^a.
\]
This verifies the claim and shows that if $\cE_M$ is the event that $X$ defines a matching plus $O(1)$ edges, then $\neg\cE_M$ is unlikely enough so that we can use \eqref{ExZ}. 
\item The sum $p(I_1+I_2+I_3)$ becomes $p(I_1+\cdots+I_d)$ which is dominated by $pI_1$ where 
\[
I_1=\int_{u=0}^1\brac{1-\frac{u^d}{d!}}^{|X|}du\approx \frac{1}{|X|^{1/d}}\int_{x=0}^\infty e^{-x^d/d!}dx=\frac{(d!)^{1/d}\G(1+1/d)}{|X|^{1/d}}.
\]
\item After this we find that \eqref{E3} becomes
\[
E(W_n\mid \cE_L,\cE_M)\approx \frac{(d!)^{1/d}\G(1+1/d)p}{(|I_1|\cdots |I_{d-1}|\cdot p)^{1/d}}
\]
\item Because the $|I_j|$ are strongly concentrated about their means, this results in replacing \eqref{E5} by
\[
E(W_n\mid \cE_M)\approx \frac{(d!)^{1/d}\G(1+1/d)}{n^{(d-1)/d}}.
\]
Multiplying by $n$ gives us part (a) of Theorem \ref{thd} with $c_d=(d!)^{1/d}\G(1+1/d)$.
\item The essential part of (b) is the inequality \eqref{compl}. For this, where $S=\set{x_{\bf i_l}:l=1,2,\ldots,s}$ is a  matching in $H$ and $m\geq n-n^{1/(d+1)}$, we use
\multstar{
\Pr(S\subseteq Y_m)\leq \sum_{i_1\leq \cdots\leq i_{s}} \Pr\brac{\bigcap_{t=1}^s\set{x_{\bf i_t}\in X_{i_t}}}=\\
 \sum_{i_1\leq \cdots\leq i_{s}}\prod_{t=1}^s\Pr\brac{x_{{\bf i_t}}\in X_{i_t}\bigg| \bigcap_{\t=1}^{t-1}\set{x_{\boldsymbol \t}\in X_{i_\t}}}\leq \sum_{i_1\leq \cdots\leq i_{s}}\prod_{t=1}^s((1+o(1))\n_{i_t}^{-(d-1)^2/d})\\
\leq \prod_{t=1}^s\sum_{l=1}^{n-m+1}\frac{1+o(1)}{(n-l+1)^{(d-1)^2/d}}\leq O\bfrac{1}{n^{(d-1)^2/(d(d+1))}}^s=O(n^{-3}),
}
for $s\geq 3d(d+1)/(d-1)^2$.

We deduce from this that we can replace \eqref{greedycost} by
\[
(c_d+o(1))\sum_{m=1}^{n-n^{1/(d+1)}}\frac{1}{(n-m+1)^{(d-1)/d}}+n^{1/(d+1)}=O(n^{1/d}).
\]
\end{enumerate}  
The final $n-n^{1/(d+1)}$ steps cost at most $d$ per step and this completes our sketch proof of Theorem \ref{thd}.
\section{Greedy for small $L$}\label{Easycase}
When $W_{i,j,k}$ is chosen uniformly from $[1,M=n^{\a}],0<\a<1$ we 
\begin{enumerate}[{\bf (a):}]
\item Let $Z_m$ denote the cost of the $m$th triple. Then for $1\leq m\leq n$ and $a\geq 1$,
\[
\Pr(\exists m:Z_m\geq a)\leq n\brac{1-\frac{a}{M}}^{(n-m+1)^2}\leq n\exp\set{-\frac{a(n-m+1)^2}{M}} \leq n^{-2},
\]
if
\beq{aval}{
a\geq \frac{3M\log n}{(n-m+1)^2}.
}
Putting $m_0=n-(3M\log n)^{1/2}$ we see that $a$ satisfies \eqref{aval} for
\[
a=\begin{cases}1&m\leq m_0.\\\rdup{\frac{3M\log n}{(n-m+1)^2}}&m>m_0.\end{cases}
\]
It follows that w.h.p. and in expectation that if $m_1=n-\log n$, then
\[
\sum_{m=1}^nZ_m\leq m_0+\sum_{m=m_0+1}^{m_1}\frac{3M\log n}{(n-m+1)^2}+M(n-m_1)=n+o(n),
\]
\end{enumerate}
\section {Greedy versus Greedy}
There is another version of the greedy algorithm where at each step we choose the ``tple'' of minimum weight that can be added to the current choice. Let $E(\l)$ denote the exponential rate $k$ random variable i.e. $\Pr(E(\l)\geq u)=e^{-\l u}$. We consider the $d$-dimensional case and argue next that if the weights $W_{i_1,\ldots,i_d}$ are independent $E(1)$ then the value of the solution given by the two algorithms is the same in distribution. So let $G_{n,1}$ be the value returned by the algorithm described above and let $G_{n,2}$ be the value returned by algorithm described in this section. We claim that $G_{n,1}$ and $G_{n,2}$ have the same distribution.

The distribution of $G_{n,1}$ is $E(n^{d-1})+G_{n-1,1}$ and the distribution of $G_{n,2}$ is $E(n^d)(1+(n-1))+G_{n-1,2}$. The term $E(n^d)(n-1)$ is a result of the fact that conditioning an exponential to be greater than $x$ is equivalent to adding $x$ to a copy of that variable. Then observe that $E(n^{d-1})=nE(n^d)$. The claim follows by induction.

Note that coincidentally, when $d=3$, $\E(G_{n,1})$ is equal to the  expected optimum value for the $d=2$ case, see \cite{NaPrSh} and \cite{Wastlund04}. This does not generalise.
\section{Final Comments}
We have analysed a random multi-dimensional assignment problem with a particular form of objective fucntion. We have shown that w.h.p. there is a simple greedy algorithm that is a $(3+o(1)$-approximation to the minimum.  It is possible to replace the 3 here by $3-\e$, by arguing that w.h.p. the optimum solution must use the (at least) second smallest $j,k$ (when $d=3$) for $\Omega(n)$ values of $i$. We omit the details as the real aim is to replace 3 by 1.

\end{document}